\documentclass{amsart}
\usepackage{amsmath, amssymb, amscd, amsthm, amsfonts}
\usepackage{tikz-cd}
\usepackage{graphicx}
\usepackage{hyperref}
\usepackage{tikz}
\usepackage{subcaption}
\usepackage{caption}
\usepackage{geometry}

\title{Knot Floer homology of boundary Dehn Twists}
\author{Rithwik Susheel Vidyarthi}
\date{}

\newtheorem{theorem}{Theorem}[section]
\newtheorem{prop}[theorem]{Proposition}
\newtheorem{lemma}[theorem]{Lemma}

\newtheorem{question}[theorem]{Question}
\newtheorem{problem}[theorem]{Problem}

\theoremstyle{definition}

\newtheorem*{acknowledgements}{\textbf{Acknowledgements}}

\begin{document}
\address{Department of Mathematics, Michigan State University, East Lansing, MI 48824, USA}
\email{vidyart2@msu.edu}

\begin{abstract}
We prove that the knot Floer complex, along with some restrictions on the flip map, determines the dual knot to $\pm 1$ surgery on the Borromean knot. In particular, this implies that Heegaard Floer homology detects boundary Dehn twists. 
\end{abstract}

\maketitle

\section{Introduction}
Ozsv\'ath and Szab\'o introduced powerful three-manifold invariants called Heegaard Floer homology \cite{OS1} that have been very influential in the field of low-dimensional topology. These gave rise to knot Floer homology, an invariant for null-homologous knots in 3-manifolds \cite{OSKnot, Ras}. Knot Floer homology can determine if a knot is fibered \cite{Ghi, NiFiber}. It also determines the genus of the knot \cite{OSgenus}. In recent years, knot Floer homology has been found to encode data about the monodromy of a fibered knot. For instance, the next-to-top grading of knot Floer homology contains information about the fixed points of the monodromy \cite{ghiggini2022knotfloerhomologyfibred, Ni, NiFixedPoints2}. In this paper we study the extent to which knot Floer homology determines the monodromy when it is a boundary Dehn twist.

The Borromean knot is described as follows. Consider the three-component Borromean link and perform 0 surgery on two of the link components. The remaining component yields a knot in $\#^2S^2\times S^1$. This is referred to as the genus one Borromean knot $\mathcal{B}_1$. The genus $g$ Borromean knot $\mathcal{B}_g$, is obtained via a connected sum of $g$ copies of $\mathcal{B}_1$. $\mathcal{B}_g$ is a fibered knot whose monodromy is the identity map on the genus $g$ surface with one boundary component $\Sigma_{g,1}$. Hedden and Watson proved that knot Floer homology detects the Borromean knot \cite{MattWatson}. Their proof follows from the observation that the knot Floer homology of $\mathcal{B}_g$ contains a single generator in the top and bottom Alexander grading with no differentials, a fact which implies that the open book decomposition $(\Sigma, \phi)$ and $(\Sigma, \phi^{-1})$, specified by any knot with this property, give rise to tight contact structures. Hence by \cite{HKM1}, $\phi$ and $\phi^{-1}$ are both right-veering which implies $\phi=\mathrm{Id}$.
\vspace{-5pt}
\begin{figure}[h]
    \centering
    \includegraphics[width=0.8\textwidth]{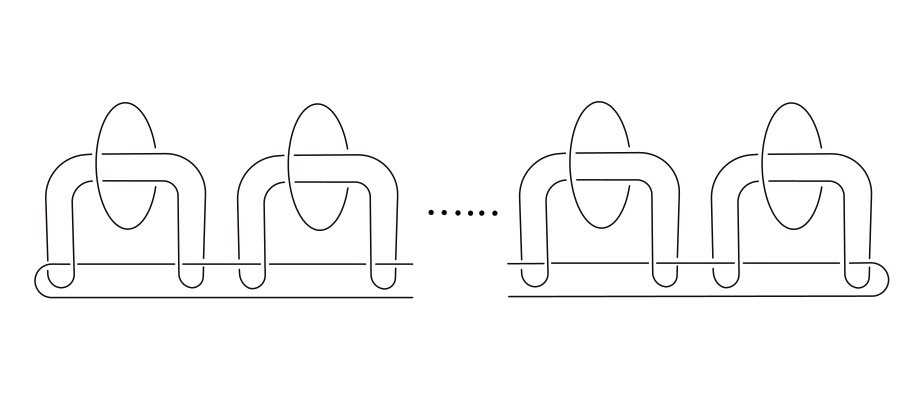}
    \caption{Genus $g$ Borromean knot $\mathcal{B}_g$}
\end{figure}

Given an element $\phi\in MCG(\Sigma_{g,1},\partial)$, we can construct the mapping torus $M_\phi$. On the torus boundary of $M_\phi$, there is a natural choice of meridian given by the circle direction and a choice of longitude given by $\partial(\Sigma_{g,1})$. This allows one to glue a solid torus by mapping the meridian of the solid torus to the longitude of the boundary torus of $M_\phi$. Gluing the solid torus in this way gives rise to a closed three-manifold $\widehat{M}_\phi$. The core of the glued-in solid torus is a knot $K_\phi$ in $\widehat{M}_\phi$. In fact, $K_\phi$ is a fibered knot whose monodromy is $\phi$. Consequently, each element of $MCG(\Sigma_{g,1},\partial)$ gives rise to pair $(\widehat{M}_\phi, K_\phi)$, where $\widehat{M}_\phi$ is a three-manifold and $K_\phi\subset \widehat{M}_\phi$ is a knot. We can therefore think of knot Floer homology as a map from $MCG(\Sigma_{g,1},\partial)$ to filtered chain complexes up to filtered chain homotopy by computing the knot Floer homology of $K_\phi$. The theorem of Hedden and Watson shows that $\mathrm{Id}\in MCG(\Sigma_{g,1},\partial)$ is the unique element mapping to $\mathrm{CFK}^\infty(\mathcal{B}_g)$, hence knot Floer homology detects $\mathrm{Id}$. There are other results in the same vein; for instance the preimage of the complexes of the unknot, trefoil and the figure eight knot are also unique, as knot Floer homology detects these knots \cite{OSgenus, Ghi}.

We perform $-1$ (respectively $+1$) surgery on the Borromean knot and consider the dual knot, that is, the core of the glued-in solid torus. The dual knot, denoted by $\mathcal{\widetilde{B}}^R_g$ (respectively $\mathcal{\widetilde{B}}^L_g$) is also fibered with the monodromy of $\mathcal{\widetilde{B}}^R_g$ (respectively $\mathcal{\widetilde{B}}^L_g$) given by a right-handed (respectively left-handed) boundary Dehn twist on $\Sigma_{g,1}$. In this paper, we calculate the full knot Floer complex of $\mathcal{\widetilde{B}}^R_g$ and $\mathcal{\widetilde{B}}^L_g$ using the dual knot formula \cite{MattLevine}. To describe the result, let $P(g,k,l)$ be a set with $2 g \choose l$ elements enumerated as $\{p^k_{l,1},\cdots,p^k_{l,{2g\choose l}}\}$
and let $Q(g,k,l)$ be a set with $2 g \choose l$ elements enumerated as $\{q^k_{l,1},\cdots,q^k_{l,{2g\choose l}}\}$. We will let $\mathbb{F}$ denote the field with two elements.

\begin{theorem}\label{MT}
At each Alexander grading $-g\leq k\leq g$, $\mathrm{CFK}^{\infty}(\mathcal{\widetilde{B}}^R_g)$ is generated over $\mathbb{F}[U,U^{-1}]$ by the basis \\ $\bigcup_{l=0}^{g-|k|-1}P(g,k,l)\bigsqcup\bigcup_{l=0}^{g-|k|}Q(g,k,l)$. 
\begin{enumerate}
    \item The Maslov grading of generators in $P(g,k,l)$ is $-g+k-k^2+1+l$.
    \item The Maslov grading of generators in $Q(g,k,l)$ is $g+k-k^2-l$.
	\item The differential of a generator $p^k_{l,j} \in P(g,k,l)$ is given as follows. 
			\[
			\partial^\infty(p^k_{l,j}) =U^{g-k-l}(q^{k+1}_{l,j} + U^{2k-1}q^{k-1}_{l,j}).
			\]
	\item The differential is zero on the subspace generated by $Q(g,k,l)$.
\end{enumerate}
\end{theorem}

Note that computations involving the surgery formula require knowledge of an auxiliary piece of data known as the flip map $\psi_{\mathrm{flip}}$. For knots in $S^3$, and L-spaces in general, $\psi_{\mathrm{flip}}$ can be determined (see \cite[Lemma 2.18]{MattLevine}). For knots in other three-manifolds, $\psi_{\mathrm{flip}}$ can be highly non-trivial. We prove the following.

\begin{prop}
\label{flipmap}
Let $\Theta_{k} \in C\{0,k\} \subseteq \mathrm{CFK}^{\infty}(\mathcal{\widetilde{B}}^R_g)$ denote the unique generator of $Q(g,k,0)$ for Alexander grading $-g\leq k\leq g-1$. Then $\psi_{\mathrm{flip}}: C\{j=0\}\to C\{i=0\}$ satisfies $\psi_{\mathrm{flip}}(U^{k}\Theta_k)=\Theta_{-k-1}$.
\end{prop}

With the above constraints on the flip-map, we conclude that the knot Floer complex determines boundary Dehn twists. More precisely we have the following.

\begin{theorem}
\label{detect}
Let $K\subset Y$ be an arbitrary knot in a three-manifold $Y$. Suppose that $\mathrm{CFK}^{\infty}(K,Y)$ is filtered chain homotopic to $\mathrm{CFK}^{\infty}(\mathcal{\widetilde{B}}^R_g)$. Furthermore, assume that the flip map on  $\mathrm{CFK}^{\infty}(K,Y)$ satisfies Proposition \ref{flipmap}. Then $K$ is isotopic to $\mathcal{\widetilde{B}}^R_g$.
\end{theorem}

We also have analogous results for $\mathcal{\widetilde{B}}^L_g$. See Theorem \ref{MT2}, Proposition \ref{flipmap2} and Theorem \ref{detect2} for the precise statements.

Since the center of $MCG(\Sigma_{g,1},\partial)$ is generated by the identity map and boundary Dehn twists, we conclude that knot Floer homology has a unique preimage for the knots corresponding to the generators of the center of $MCG(\Sigma_{g,1},\partial)$. It is natural to ask if knot Floer homology detects the full center of $MCG(\Sigma_{g,1},\partial)$; that is, if knot Floer homology detects powers of boundary Dehn twists. These correspond to fibered knots given by the dual knots to $\pm \frac{1}{n}$ surgery on $\mathcal{B}_g$.
\begin{question}
Does knot Floer homology detect the center of $MCG(\Sigma_{g,1},\partial)$?
\end{question} 
In a different direction, we can study surfaces with more than one boundary component and possibly punctures. Let $\Sigma_{g,b,n}$ denote the genus $g$ surface with $b$ boundary components and $n$ punctures. Having multiple boundary components corresponds to a fibered link $L$ in a three-manifold $Y$. Additionally, if the surface has punctures, this data can be thought of as a transverse link (the orbit of the punctures) in a contact three-manifold (induced by the open book decomposition of the fibered link $L$). We can now study whether link/knot Floer homology detects elements in the mapping class group of $\Sigma_{g,b,n}$.

\begin{question}
What elements of $MCG(\Sigma_{g,n,b},\partial)$ are detected by link/knot Floer homology?
\end{question} 

A way to strengthen Theorem \ref{detect} is to remove the condition on $\psi_{\mathrm{flip}}$. It turns out that there are only two possible choices of $\psi_{\mathrm{flip}}$ on $\mathrm{CFK}^{\infty}(\mathcal{\widetilde{B}}^R_g)$. In order to remove the constraint on $\psi_{\mathrm{flip}}$ it suffices to show that only one of the choices can be realized geometrically.

Note that we have only calculated a part of the flip-map of $\mathcal{\widetilde{B}}^R_g$ and $\mathcal{\widetilde{B}}^L_g$. It remains to calculate the full-flip map. It would also be worthwhile to compute the flip-map of $\mathcal{\widetilde{B}}^R_g$ and $\mathcal{\widetilde{B}}^L_g$ over $\mathbb{Z}$ coefficients. A description of the flip map of $\mathcal{B}_g$ over $\mathbb{Z}$ coefficients given in \cite[Section 3.2]{MarkJabuka}.

\begin{problem}
Calculate $\psi_{\mathrm{flip}}$ of $\mathcal{\widetilde{B}}^R_g$ and $\mathcal{\widetilde{B}}^L_g$ over $\mathbb{Z}$ coefficients.
\end{problem}

\begin{acknowledgements}
I would like to thank my advisor, Matt Hedden, for his invaluable guidance and patience answering all my questions. This project was supported by the grant DMS-2104664 and RTG: Algebraic and Geometric Topology at Michigan State University DMS-2135960.

AI usage in this project has been minimal, restricted to editing concerns such as grammatical errors and spellings. 

\end{acknowledgements}

\section{Knot Floer Homology of $\mathcal{B}_g$}\label{section-knot floer}
Oszv\'ath and Szab\'o calculated the knot Floer complex of the Borromean knot and its flip map in \cite{OSKnot}. They prove that $\mathrm{CFK}^{\infty}(\mathcal{B}_g)=\mathbb{F}[U, U^{-1}] \otimes_\mathbb{Z} \Lambda^*H^1(\Sigma_g;\mathbb{Z})$. Following \cite{integer}, we denote the $\mathbb{Z}\oplus \mathbb{Z}$ filtration as follows.
$$C\{i,j\}=U^{-i}\otimes \Lambda^{g-i+j}H^1(\Sigma_g;\mathbb{Z}).$$
Below we draw the generating sets for the knot Floer complexes for the genus 1 and 2 Borromean knots over $\mathbb{F}[U,U^{-1}]$.\newline
\begin{minipage}{0.45\textwidth}
\label{CFKBorromean}
 \centering
   \begin{tikzpicture}
    \draw[thick, <->] (-3,0) -- (3,0) node[anchor=north west] {$i$}; 
    \draw[thick, <->] (0,-3) -- (0,3) node[anchor=south east] {$j$}; 

    \draw[step=1cm, gray, very thin] (-3,-3) grid (3,3);

    \fill[black] (0.3,0) circle (3pt);
    \fill[black] (-0.3,0) circle (3pt);
    \fill[black] (0,1) circle (3pt);
    \fill[black] (0,-1) circle (3pt);
  \end{tikzpicture} 
 \\[5pt]
 
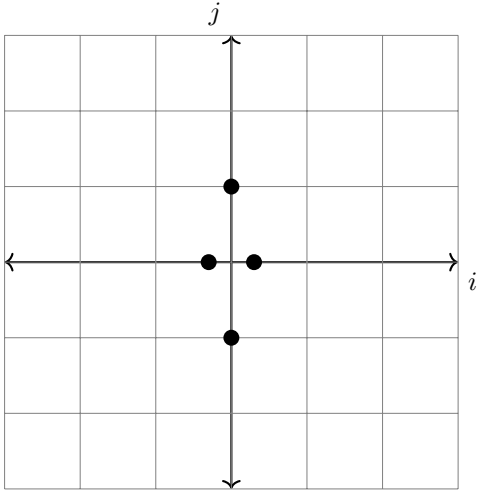
\captionof{figure}{$\mathrm{CFK}^{\infty}(\mathcal{B}_1)$}
\end{minipage}
\hspace{1cm}
\begin{minipage}{0.45\textwidth}
\centering
\begin{tikzpicture}
    \draw[thick, <->] (-3,0) -- (3,0) node[anchor=north west] {$i$}; 
    \draw[thick, <->] (0,-3) -- (0,3) node[anchor=south east] {$j$}; 

    \draw[step=1cm, gray, very thin] (-3,-3) grid (3,3);

    \fill[black] (0.1,0.1) circle (2pt);
    \fill[black] (-0.1,0.1) circle (2pt);
    \fill[black] (0.1,-0.1) circle (2pt);
    \fill[black] (0.3,0) circle (2pt);
    \fill[black] (-0.3,-0) circle (2pt);
    \fill[black] (-0.1,-0.1) circle (2pt);

     \fill[black] (0.1,1.1) circle (2pt);
     \fill[black] (-0.1,1.1) circle (2pt);
      \fill[black] (0.1,0.9)coordinate(top) circle (2pt);
       \fill[black] (-0.1,0.9) circle (2pt);
         
      \fill[black] (0.1,-1.1)  circle (2pt);
      \fill[black] (-0.1,-1.1) circle (2pt);
      \fill[black] (0.1,-0.9) circle (2pt);
       \fill[black] (-0.1,-0.9)  coordinate(bottom) circle (2pt);
    
    \fill[black] (0,2) coordinate(topmost) circle (2pt);
    \fill[black] (0,-2) coordinate(bottommost) circle (2pt);
    
         
      \fill[black] (1.1,-0.1) circle (2pt);
      \fill[black] (0.9,-0.1) circle (2pt);
      \fill[black] (1.1,0.1) circle (2pt);
       \fill[black] (0.9,0.1) coordinate(right) circle (2pt);


     \fill[black] (-0.9,0.1) circle (2pt);
     \fill[black] (-1.1,0.1) circle (2pt);
      \fill[black] (-0.9,-0.1)coordinate(left) circle (2pt);
       \fill[black] (-1.1,-0.1) circle (2pt);
       
	\fill[black] (2,0) coordinate(rightmost) circle (2pt);
     \fill[black] (-2,0) coordinate(leftmost) circle (2pt);
       
    \draw[Stealth-, blue, very thin, shorten >=3pt, shorten <=3pt] (left) to[bend right=20] (bottom);
    \draw[-Stealth, blue, very thin, shorten >=3pt, shorten <=3pt] (top) to[bend left=20] (right);
    
     \draw[-Stealth,blue, very thin, shorten >=3pt, shorten <=3pt] (topmost) to[bend left=20] node[midway, right, font=\small] {$\psi_{\mathrm{flip}}$}  (rightmost);
     \draw[Stealth-,blue, very thin, shorten >=3pt, shorten <=3pt] (leftmost) to[bend right=20] (bottommost);

\end{tikzpicture}
\\[5pt]
 
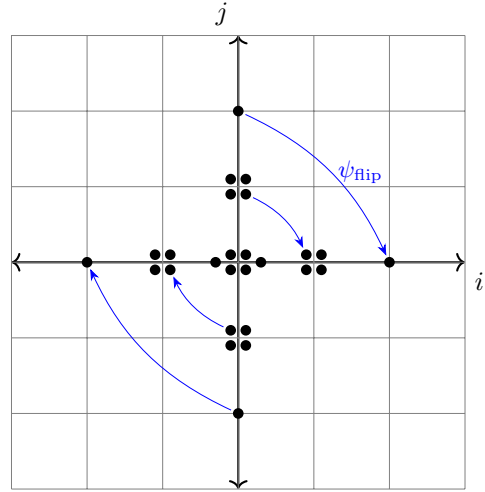
\captionof{figure}{$\mathrm{CFK}^{\infty}(\mathcal{B}_2)$ along with its flip-map $\psi_{\mathrm{flip}}$ drawn in blue}
\end{minipage}
\vspace{3pt}
\subsection{Flip map of Borromean knot}
In this section, we will describe the flip map of the Borromean knot, as given in \cite{integer}. We will work with $\mathbb{F}=\frac{\mathbb{Z}}{2\mathbb{Z}}{•}$ coefficients throughout. Let $\{A_i, B_i\}_{i=1}^g$ be the dual basis to the standard symplectic basis of $H_1(\Sigma_g)$. Then $H^1(\Sigma_g)=\mathbb{F}\langle A_1,B_1,A_2,B_2, \cdots, A_g,B_g\rangle$. We have the map:
\begin{align*}
    I: H^1(\Sigma)&\to H^1(\Sigma); \\
   I(A_i) & = B_i, \\
   I(B_i) & = A_i.
\end{align*}
$I$ induces a map on $\Lambda^kH^1(\Sigma)$ by letting it commute with the wedge product. Let $*:\Lambda^kH^1(\Sigma)\to \Lambda^{2g-k}H^1(\Sigma)$ be the Hodge star operator. The flip map is as follows.
$$\psi_{\mathrm{flip}}: C\{i,j\}\to C\{j,i\};$$
$$\psi_{\mathrm{flip}}(\omega\otimes U^{-i})=(*I\omega)\otimes U^{-j}.$$
We will start with a sample calculation for the genus 2 Borromean knot.
   \begin{align*}
       \Lambda^0H^1&\cong\mathbb{F}. \\
       \Lambda^1H^1&\cong\mathbb{F}^4=\langle A_1,B_1,A_2,B_2\rangle .\\
       \Lambda^1H^2&\cong\mathbb{F}^6=\langle A_1\wedge B_1,A_2\wedge B_2,A_1\wedge A_2, B_1\wedge B_2,A_1\wedge B_2, B_1\wedge A_2\rangle .\\
       \Lambda^1H^3&\cong\mathbb{F}^4=\langle A_1\wedge B_1\wedge A_2,B_1\wedge A_2\wedge B_2,A_1\wedge B_1\wedge B_2,A_1\wedge A_2\wedge B_2\rangle .\\
       \Lambda^1H^4&\cong\mathbb{F}=\langle A_1\wedge B_1\wedge A_2\wedge B_2\rangle .\\
   \end{align*} 
With the above basis, we write down the maps $I$ and $*$ in Table 1.
\begin{table}[h!]
\centering
\begin{tabular}{||c c c ||}
\hline
   Basis element  & $I$ & $*$  \\
\hline 
1 & 1 & $A_1\wedge B_1\wedge A_2\wedge B_2$ \\
$A_1$ & $B_1$ & $B_1\wedge A_2 \wedge B_2$  \\
$B_1$ & $A_1$ & $A_1\wedge A_2 \wedge B_2$ \\
$A_2$ & $B_2$ & $A_1\wedge B_1 \wedge B_2$  \\
$B_2$ & $A_2$ & $A_1\wedge B_1 \wedge A_2$ \\
$A_1\wedge B_1$ & $A_1\wedge B_1$ & $A_2\wedge B_2$ \\
$A_2\wedge B_2$ & $A_2\wedge B_2$ & $A_1\wedge B_1$ \\
$A_1\wedge A_2$ & $B_1\wedge B_2$ & $B_1\wedge B_2$ \\
$B_1\wedge B_2$ & $A_1\wedge A_2$ & $A_1\wedge A_2$ \\
$A_1\wedge B_2$ & $B_1\wedge A_2$ & $B_1\wedge A_2$ \\
$B_1\wedge A_2$ & $A_1\wedge B_2$ & $A_1\wedge B_2$ \\
$A_1\wedge B_1 \wedge A_2$ & $B_1\wedge A_1\wedge B_2$ & $B_2$\\
$A_1\wedge B_1 \wedge B_2$ & $B_1\wedge A_1\wedge A_2$ & $B_1$\\
$A_1\wedge A_2 \wedge B_2$ & $B_1\wedge A_2\wedge B_2$ & $A_2$\\
$B_1\wedge A_2 \wedge B_2$ & $A_1\wedge A_2\wedge B_2$ & $A_1$\\
$A_1\wedge B_1\wedge A_2\wedge B_2$ & $A_1\wedge B_1\wedge A_2\wedge B_2$ & 1 \\
\hline
\end{tabular}
\caption{I and * maps for Genus 2 }
\end{table}
Then we calculate the flip map. For $j=-2$, $$\psi_{\mathrm{flip}}: C\{0,-2\}\cong \Lambda^0\mathbb{F}^4\to C\{-2,0\}\cong U^{2}\otimes \Lambda^4\mathbb{F}^4 ;$$
$$1\mapsto U^2\otimes A_1\wedge B_1\wedge A_2\wedge B_2 .$$
For $j=-1$, $$\psi_{\mathrm{flip}}: C\{0,-1\}\cong \Lambda^1\mathbb{F}^4\to C\{-1,0\}\cong U^{1}\otimes \Lambda^3\mathbb{F}^4  ;$$
$$A_1\mapsto U\otimes(A_1\wedge A_2\wedge B_2),$$
$$B_1\mapsto  U\otimes(B_1\wedge A_2\wedge B_2),$$
$$A_2\mapsto  U\otimes(A_1\wedge B_1\wedge A_2),$$
$$B_2\mapsto  U\otimes(A_1\wedge B_1\wedge B_2).$$
For $j=0$, $$\psi_{\mathrm{flip}}: C\{0,0\}\cong \Lambda^2\mathbb{F}^4\to C\{0,0\}\cong U^{0}\otimes \Lambda^2\mathbb{F}^4 ;$$
$$A_1\wedge B_1\mapsto A_2\wedge B_2,$$
$$A_2\wedge B_2\mapsto A_1\wedge B_1,$$ 
$$A_1\wedge A_2\mapsto A_1\wedge A_2,$$ 
$$B_1\wedge B_2\mapsto B_1\wedge B_2,$$ 
$$A_1\wedge B_2\mapsto A_1\wedge B_2,$$ 
$$B_1\wedge A_2\mapsto B_1\wedge A_2.$$
For $j=1$, $$\psi_{\mathrm{flip}}: C\{0,1\}\cong \Lambda^3\mathbb{F}^4\to C\{1,0\}\cong U^{-1}\otimes \Lambda^1\mathbb{F}^4;$$
$$A_1\wedge B_1\wedge A_2 \mapsto U^{-1}\otimes A_2,$$
$$A_1\wedge B_1\wedge B_2\mapsto  U^{-1}\otimes B_2,$$
$$A_1\wedge A_2\wedge B_2\mapsto  U^{-1}\otimes A_1,$$
$$B_1\wedge A_2\wedge B_2\mapsto  U^{-1}\otimes B_1.$$
For $j=2$, $$\psi_{\mathrm{flip}}: C\{0,2\}\cong \Lambda^4\mathbb{F}^4\to C\{2,0\}\cong U^{-2}\otimes \Lambda^0\mathbb{F}^4;$$
$$A_1\wedge B_1\wedge A_2\wedge B_2 \mapsto U^{-2}\otimes1.$$

We will now modify the labelling of generators so that the flip-map has a nice symmetric description. This description will make it easier to use the dual knot formula. We prove the following. 

\begin{lemma}\label{basis}

    For $ i > 0$, there is a labelling of generators of $C\{0,i\}=\{x_j\}_{j=1}^{2g\choose {g-i}}$, $C\{0,-i\}=\{\widetilde{x}_j\}_{j=1}^{2g\choose {g-i}}$, and $C\{0,0\}=\{y_j\}_{j=1}^{\frac{1}{2}{2g\choose g}-2^{g-1}}\cup \{\widetilde{y}_j\}_{j=1}^{\frac{1}{2}{2g\choose g}-2^{g-1}} \cup \{f_j\}_{j=1}^{2^g}$ such that the flip-map of $\mathcal{B}_g$ is given as follows.
    \begin{align*}
        \psi_{\mathrm{flip}}:C\{0,i\}&\to C\{i,0\}\cong U^{-i}C\{0,-i\}; \\
        \psi_{\mathrm{flip}}(x_j)&=U^{-i}\widetilde{x}_j, \\
        \psi_{\mathrm{flip}}(\widetilde{x}_j)&=U^{i}x_j, \\
        \psi_{\mathrm{flip}}(y_j)&=\widetilde{y}_j,\\
        \psi_{\mathrm{flip}}(\widetilde{y}_j)&=y_j,\\
        \psi_{\mathrm{flip}}(f_j)&=f_j,
    \end{align*}
for all $i \in \mathbb{Z}$.

\end{lemma}

\begin{proof}

Let $A_i$ and $B_i$ be duals to the standard symplectic basis of $H_1(\Sigma;\mathbb{Z})$. Then $\{ A_1,B_1,\ldots, A_g,B_g\}$ is a basis for $H^1(\Sigma_g;\mathbb{Z})$. Since $*I*I=\mathrm{Id}$, we see that by a suitable reordering of $\{ A_1,B_1,\ldots, A_g,B_g\}$, $\psi_{\mathrm{flip}}$ is realized by the $i$-$j$ symmetry when $i\neq j$. More concretely, for $ i>0$, pick a basis for $C\{0,i\}\cong \Lambda^{g+i}H^1(\Sigma_g)$ arising from the symplectic basis. Since $\psi_{\mathrm{flip}}$ is a quasi-isomorphism, a basis element $x\in C\{0,i\}$ defines a basis element for $C\{0,-i\}$ by $\widetilde{x}:=\psi_{\mathrm{flip}}(x)U^i$. Subsequently we have the following.
$$\psi_{\mathrm{flip}}:C\{i,0\}\to C\{0,i\}\cong U^{-i}C\{0,-i\};$$
$$\psi_{\mathrm{flip}}(x)=U^{-i}\widetilde{x}.$$
At $C\{0,0\}=\Lambda^g \mathbb{Z}^{2g}$, the flip map is more subtle. Considering the generating set of $\Lambda^g\mathbb{Z}^{2g}$ coming from the symplectic basis, a generator is fixed by $\psi_{\mathrm{flip}}$ iff it does not contain a term of the form $A_i\wedge B_i$. In other words, $x_1\wedge \ldots\wedge x_g \in \Lambda^g\mathbb{Z}^{2g}$ is fixed by  $\psi_{\mathrm{flip}}$ iff each $x_i$ is either $A_i$ or $B_i$. There are $2^g$ elements of this form. The remaining ${2g \choose g} - 2^g$ elements are swapped among themselves. We conclude that $C\{0,0\}=L\cup \widetilde{L}\cup F$, where $\psi_{\mathrm{flip}}$ sends elements of $L$ to $\widetilde{L}$(and vice-versa) and acts as the identity on $F$.
\end{proof}

\section{Dual Knot}\label{dual knot}
In this section, we shall utilize the dual knot formula of Hedden-Levine \cite{MattLevine} to calculate the knot Floer complex of $\mathcal{\widetilde{B}}^R_g$. The calculations for $\mathcal{\widetilde{B}}^L_g$ are analogous. Let $K$ be a null-homologous knot in a three-manifold $Y$ and let $K^\nu$ denote the dual knot of $-1$ surgery on $K$. The complexes $A^\infty_s$ and $B^\infty_s$ for $s\in \mathbb{Z}$ denote a copy of $\mathrm{CFK}^\infty(Y,K)$. $A^\infty_s$ and $B^\infty_s$ have two filtrations as given in \cite[Page 232]{MattLevine}. Substituting $d=1$, $k=-1$ and $s_l=l$ we get the following.
\begin{itemize}
    \item For $[x,i,j]\in A^\infty_s: \mathcal{I}[x,i,j]=max\{i,j-s\}$ and $\mathcal{J}[x,i,j]=max\{i-s,j-2s+1\}$
    \item For $[x,i,j]\in B^\infty_s: \mathcal{I}[x,i,j]=i$ and $\mathcal{J}[x,i,j]=i-s$
\end{itemize}

We use the convention that generators of the complex $A^\infty_i$ are labeled with subscript $i$, e.g $a_i \in A^\infty_i$, and generators of complex $B^\infty_i$ are labeled with subscript $i$ and superscript $'$, e.g $b_i'\in B^\infty_i$. Let $A_s$ (respectively $B_s$) denote the subcomplex of $A_s^\infty$ (respectively $B_s^\infty$) with $\mathcal{I}=0$. We have the following two maps.
\begin{align*}
    v_k&:A_k\to B_k,\\
    h_k& : A_k \to B_{k-1}.
\end{align*}
Here $v_k$ is the projection map and $h_k$ is an induced map by the flip-map. More precisely, $h_k(x)=U^k\psi_{\mathrm{flip}}(x)$. The generators of $K^\nu$ in Alexander grading $-k$ are in one-to-one correspondence with the generators in the homology of the mapping cone $MC(k)=\mathrm{Cone}(h_{k+1}\oplus v_k: A_k \oplus A_{k+1}\to B_{k})$ for each $-g\leq k\leq g$. The differentials are determined by the induced map of $\partial^\infty$. Let $\widetilde{gr}$ and $gr$ denote the Maslov grading of the generators of $K$ and $K^\nu$ respectively. We can calculate the Maslov grading by the following formula in \cite[Page 232]{MattLevine}.
\begin{itemize}
    \item If $x\in A^\infty_k$: $gr(x)=\widetilde{gr}(x)-k-k^2+1$
    \item If $y\in B^\infty_k$: $gr(y)=\widetilde{gr}(y)-k-k^2$
\end{itemize}

\begin{proof}[Proof of Theorem \ref{MT}]
From the mapping cone, we see that the Alexander gradings of the dual knot are between $-g$ and $g$. In Alexander grading $g$, there is a unique generator $x$ with $\partial^\infty(x)=0$. Next, we consider Alexander grading $g-1$. We observe that the homology of the mapping cone $MC(-g+1)$ consists of ${2g \choose0} +{2g\choose 1}$ generators $\in  B_{-g+1}\{i=0, g-1 \leq j\leq g\}$ belonging to the top two gradings. The differential vanishes on these generators. There is one generator of the form $x+y$ where $x\in A_{-g+1}$ and $y\in A_{-g+2}$ with $\partial^\infty (x+y)=h_{-g+1}(x)+v_{-g+2}(y)$. See Figure \ref{g-1}. 

\begin{figure}[h!]
\centering
\begin{tikzpicture}

\node[label=south:$B_{-g+1}$] (B2) at (-1,1) {
\begin{tikzpicture}[scale=1]
    \filldraw[teal!35] (0, -1.75) rectangle (1, 3.75);

    \begin{scope}[thin, black!50!white]
        \draw[<->] (-1.5, 0.5) -- (2.5, 0.5);
        \draw[<->] (0.5, -2) -- (0.5, 4);
    \end{scope}

    \node at (-0.5,-1.75) {\small $g-1$};
    
    \node[red] at (2.8,3.25)	{\small $\binom{2g}{0}$ elements};
    \node[red] at (2.8,2.25)	{\small $\binom{2g}{1}$ elements};
    
    \draw[-Stealth, red, very thin, shorten >=4pt, shorten <=4pt] (2,3.3) to[bend right=30] (0.5,3.3);
    \draw[-Stealth, red, very thin, shorten >=4pt, shorten <=4pt] (2,2.3) to[bend right=30] (0.5,2.3);
    
    \fill (0.5,3.25) circle (3pt);
    
    \fill (0.1,2.25) circle (3pt);
    \fill (0.9,2.25) circle (3pt);
    \foreach \x in {0.3, 0.4, 0.5, 0.6, 0.7} {
    \fill[black] (\x, 2.25) circle (0.7pt);}

\end{tikzpicture}
};

\node[label=north:$A_{-g+1}$] (A1) at (-5,9) {
\begin{tikzpicture}[scale=1]
    \filldraw[teal!35] (0, -2) rectangle (1, -1);
    \filldraw[black!35!white] (-1.5, 0) rectangle (1, -1);

    \begin{scope}[thin, black!50!white]
        \draw[<->] (-1.5, 2) -- (2.5, 2);
        \draw[<->] (0.5, -2.5) -- (0.5, 3);
    \end{scope}

    \node at (-0.50,-2) {\small $g-1$};
    \node at (-1.75,-0.5) {\small $g$};
    
    \fill  (0.5,-1.5) circle (3pt) node [right, font=\small] {$x$};

\end{tikzpicture}
};

\node[label=north:$A_{-g+2}$] (A2) at (3,9) {
\begin{tikzpicture}[scale=1]
    \filldraw[black!35!white] (0, -2) rectangle (1, 0);
    \filldraw[teal!35] (-1.5, 1) rectangle (1, 0);

    \begin{scope}[thin, black!50!white]
        \draw[<->] (-1.5, 2) -- (2.5, 2);
        \draw[<->] (0.5, -2.5) -- (0.5, 3);
    \end{scope}

    \node at (-0.50,-2) {\small $g-2$};
    \node at (-2,0.5) {\small $g-1$};

    \fill (-1.25,0.5) circle (3pt) node[right, font=\small] {$y$};

\end{tikzpicture}
};

		\draw[->] (A1) to node[pos=0.5, left]{$v_{-g+1}$} (B2);
        \draw[->] (A2) to node[pos=0.5, left]{$h_{-g+2}$} (B2);
        
\end{tikzpicture}
\caption{Mapping cone for Alexander grading $g-1$ highlighted in teal. We have only drawn the generators that survive the mapping cone.}
\label{g-1}
\end{figure}
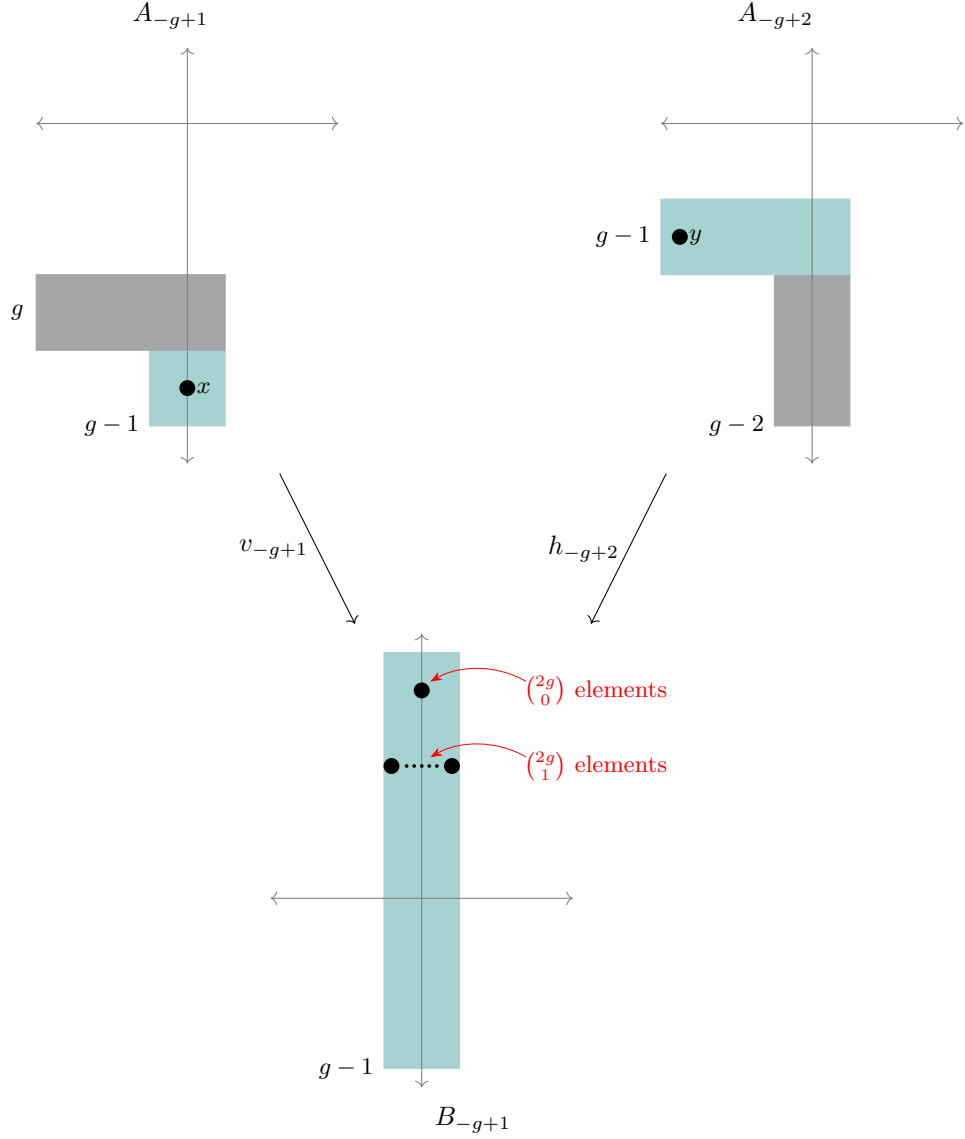

Moving on to Alexander grading $g-2$, we make a similar observation. The homology of the mapping cone $MC(-g+2)$ consist of two sets of generators. There are  ${2g \choose0} + {2g\choose 1}+{2g\choose 2}$ generators in the top three gradings of  $B_{-g+2}\{i=0, g-2 \leq j \leq g\}$. $\partial^\infty$ vanishes on these generators. There are ${2g\choose 0}+{2g\choose 1}$ generators $\in A_{-g+2}\oplus A_{-g+3}$ satisfying $\partial^\infty\neq 0$. It is straightforward to generalize this to Alexander grading $k$, where $0\leq k \leq g$. Figure \ref{mappingcone} below shows the mapping cone. 

\begin{figure}[h!]
\centering
\begin{tikzpicture}

\node[label=south:$B_{-k-1}$] (B1) at (-5,2) {
\begin{tikzpicture}[scale=1]
    \filldraw[black!35!white] (0, -2) rectangle (1, 3);

    \begin{scope}[thin, black!50!white]
        \draw[<->] (-1.5, 1.5) -- (2.5, 1.5);
        \draw[<->] (0.5, -2) -- (0.5, 3);
    \end{scope}

    \node at (-0.5,-2) {\small $k+1$};

    \fill (-1,0) circle (3pt);
    \node[below,font=\footnotesize] at  (-1,0) {$h_{-k}(x^k_{l,m})$};
\end{tikzpicture}
};

\node[label=south:$B_{-k}$] (B2) at (0,2) {
\begin{tikzpicture}[scale=1]
    \filldraw[teal!35] (0, -2) rectangle (1, 3);

    \begin{scope}[thin, black!50!white]
        \draw[<->] (-1.5, 1.5) -- (2.5, 1.5);
        \draw[<->] (0.5, -2) -- (0.5, 3);
    \end{scope}

    \node at (-0.25,-2) {\small $k$};
    
    \fill (0.5,0) circle (3pt) node [right,font=\small] {$z$};

\end{tikzpicture}
};

\node[label=south:$B_{-k+1}$] (B3) at (5,2) {
\begin{tikzpicture}[scale=1]
    \filldraw[black!35!white] (0, -2) rectangle (1, 3);

    \begin{scope}[thin, black!50!white]
        \draw[<->] (-1.5, 1.5) -- (2.5, 1.5);
        \draw[<->] (0.5, -2) -- (0.5, 3);
    \end{scope}

    \node at (-0.5,-2) {\small $k-1$};
     \fill (-0.5,1) circle (3pt);
     \node[below,font=\footnotesize] at  (-1,1) {$v_{-k+1}(y^k_{l,m})$};
\end{tikzpicture}
};

\node[label=north:$A_{-k}$] (A2) at (-2.5,10) {
\begin{tikzpicture}[scale=1]
    \filldraw[teal!35] (0, -2) rectangle (1, 1);
    \filldraw[black!50!white] (-1.5, 0) rectangle (1, 1);

    \begin{scope}[thin, black!50!white]
        \draw[<->] (-1.5, 1.5) -- (2.5, 1.5);
        \draw[<->] (0.5, -2) -- (0.5, 3);
    \end{scope}

    \node at (-0.50,-1.5) {\small $k$};
    \node at (-2.2,0) {\small $k+1$};
    
    \fill (0.5,-0.5) circle (3pt);
    \node[below,font=\footnotesize] at  (1,-0.5) {$x^k_{l,m}$};
    
\end{tikzpicture}
};

\node[label=north:$A_{-k+1}$] (A3) at (2.5,10) {
\begin{tikzpicture}[scale=1]
    \filldraw[black!35!white] (0, -2) rectangle (1, 1);
    \filldraw[teal!35] (-1.5, 0.5) rectangle (1, 1.5);

    \begin{scope}[thin, black!50!white]
        \draw[<->] (-1.5, 1.5) -- (2.5, 1.5);
        \draw[<->] (0.5, -2) -- (0.5, 3);
    \end{scope}

    \node at (-0.50,-1.5) {\small $k-1$};
    \node at (-2,0.5) {\small $k$};

    \fill (-0.5,1) circle (3pt);
     \node[below,font=\footnotesize] at  (-0.5,1) {$y^k_{l,m}$};
\end{tikzpicture}
};

		\draw[->] (A2) to node[pos=0.5, left]{$h_{-k}$} (B1);
        \draw[->] (A3) to node[pos=0.5, left]{$h_{-k+1}$} (B2);
		\draw[->] (A2) to node[pos=0.5, left]{$v_{-k}$} (B2);
		\draw[->] (A3) to node[pos=0.5, left]{$v_{-k+1}$} (B3);
\end{tikzpicture}
\caption{Mapping cone for Alexander grading $k$ highlighted with teal. Observe that $v_{-k}(x^k_{l,m})=z=h_{-k+1}(y^k_{l,m})$}
\label{mappingcone}
\end{figure}
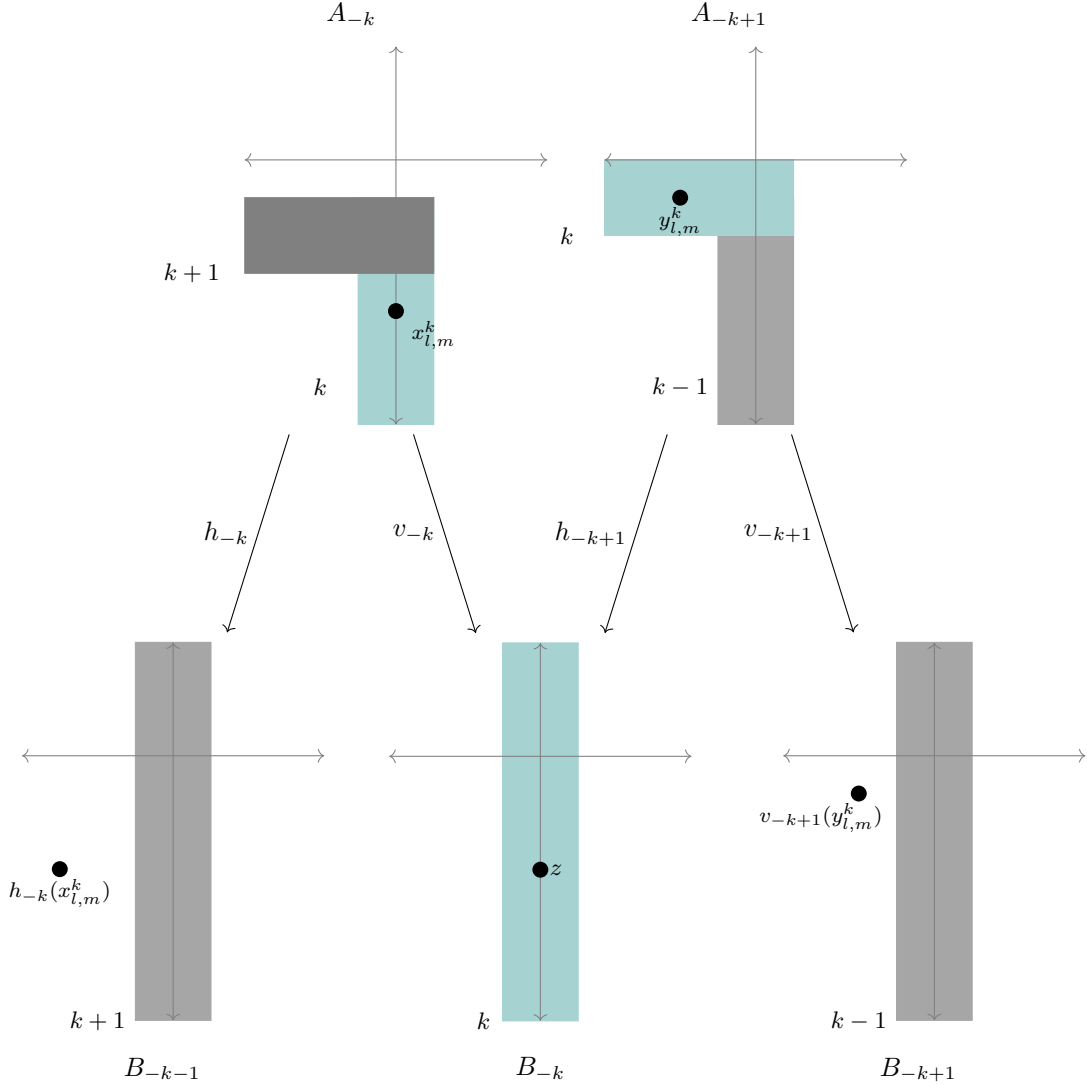
In the homology of the mapping cone $MC(-k)$, there are $\sum_{l=0}^{g-k} \binom{2g}{l}$ generators in $B_{-k}$ with $\partial^\infty=0$. These are the generators at the top of the vertical strip $B_{-k}\{i=0,k \leq j \leq g\}$. Let $Q(g,k,l)$ be the set containing the $2g\choose l$ generators in $B_{-k}\{i=0,j=g-l\}$ for $0\leq l \leq g-k$. We enumerate $Q(g,k,l)$ as $\{q^k_{l,1},\cdots,q^k_{l,{2g\choose l}}\}$. There are also $\sum_{l=0}^{g-k-1} \binom{2g}{l}$ generators in $MC(-k)$ with $\partial^\infty\neq 0$. These are in \[ A_{-k}\{i=0,-g\leq j\leq -k-1\}\oplus A_{-k+1}\{-g-k+1 \leq i \leq -2k,j=-k+1\} .\] For $0\leq l\leq g-k-1$, let $P(g,k,l)$ be a set containing the $2g\choose l$ generators in \[A_{-k}\{i=0, j=-g+l\}\oplus A_{-k+1}\{i=-g-k+1+l,j=-k+1\} \] that are of the form $p^k_{l,m}=x^k_{l,m}+y^k_{l,m}$ where $x^k_{l,m}\in A_{-k}\{i=0, j=-g+l\}$ and $y^k_{l,m}\in A_{-k+1}\{i=-g-k+1+l,j=-k+1\}$. We enumerate $P(g,k,l)$ as $\{p^k_{l,1},\cdots,p^k_{l,{2g\choose l}}\}$. The only non-trivial term of $\partial^\infty$ on $P(g,k,l)$ is $h_{-k}(x^k_{l,m})+v_{-k+1}(y^k_{l,m})$.

Now that we have the set of generators, let us understand the differentials, which is made up of the $h$ and $v$ maps. If $x\in A_s$ is at $(0,t)$, then $h_s(x)=U^{s-t}\widetilde{x}'\in B_{s-1}$. Consider the element $x_{l,m}^{k}\in A_{-k}(i=0,j=-g+l)$, then $h_{-k}(x_{l,m}^{k})=U^{(-k)-(-g+l)}\widetilde{x}'=U^{g-k-l}\widetilde{x}'$ where $\widetilde{x}'$ has Alexander grading $k+1$. There are $2g \choose l$ such generators. Moreover, $\widetilde{x}'$ is a generator in $B_{-k-1}$ that survived in the mapping cone $MC(-k-1)$, i.e, $\widetilde{x}'$ is an element of $Q(g,k+1,l)$.

To analyze the $v$-map, we need to understand the image of $v_{-k+1}$ on $A_{-k+1}(-g-k+1 \leq i \leq -2k,j=-k+1)$. Let  $y^k_{l,m} \in A_{-k}(i=-g-k+1+l,j=-k+1)$.  Then we can calculate the differential, $\partial^\infty (y^k_{l,m})=U^{g+k-1-l}z$ where $z$ now has Alexander grading $k-1$ and in fact $z\in Q(g,k-1,l)$. Therefore in Alexander grading $k$ the differential is as follows.
\begin{align*}
    \partial^\infty(p_{l,m}^k)&=\partial^\infty(x_{l,m}^{k}+y^k_{l,m})\\
                        &=h_{-k}(x^k_{l,m})+v_{-k+1}(y^k_{l,m})\\
                        &=U^{g-k-l}q_{l,m}^{k+1}+U^{g+k-l-1} q_{l,m}^{k-1}
\end{align*}
where $m=1,\ldots, {2g\choose l}$ and $0\le l \le g-k-1$. The analysis is similar for Alexander grading $-k$, and we define $P(g,-k,l)$ and $Q(g,-k,l)$ as before. The differentials are given as follows.
\[
\partial^\infty(p_{l,m}^{-k})=U^{g-(-k)-l}q_{l,m}^{-k+1}+U^{g+(-k)-l-1} q_{l,m}^{-k-1}
\]
where $m=1,\ldots, {2g\choose l}$ and $0\le l \le g-(-k)-1$.

Using the formula for Maslov gradings we conclude the following.
\begin{itemize}
    \item The Maslov grading of elements in $P(g,k,l)$ is $-g+k-k^2+1+l$.
    \item The Maslov grading of elements in $Q(g,k,l)$ is $g+k-k^2-l$.
\end{itemize}
This completes the proof of the theorem.
\end{proof}

We can use the same methods to calculate the dual knot of $+1$ surgery to $\mathcal{{B}}_g$. For brevity, we omit the details and write the statement below.
\begin{theorem}
\label{MT2}
At each Alexander grading $-g\leq k\leq g$, $\mathrm{CFK}^{\infty}(\mathcal{\widetilde{B}}^L_g)$ is generated over $\mathbb{F}[U,U^{-1}]$ by the basis \\ $\bigcup_{l=0}^{g-|k|}P(g,k,l)\bigsqcup\bigcup_{l=0}^{g-|k|-1}Q(g,k,l)$.
\begin{enumerate}
\item The Maslov grading of generators in $P(g,k,l)$ is $-g+k+k^2+l$.
\item The Maslov grading of generators in $Q(g,k,l)$ is $g+k+k^2-1-l$.
\item The differential of a generator $p^k_{l,j} \in P(g,k,l)$ is given as follows. 
		\[
		\partial^\infty(p^k_{l,j}) =U^{g-k-l}(U^{2k+1}q^{k+1}_{l,j} + q^{k-1}_{l,j})
		\]
\item  The differential is zero on the subspace generated by $Q(g,k,l)$
\end{enumerate}
\end{theorem}

\section{Flip map of the dual Borromean knot}\label{flip_dual}

The flip map $\psi_{\mathrm{flip}}$ is an invariant of the knot up to filtered chain homotopy equivalence. When restricted to the subcomplex $C\{i=0\}$, the flip map gives a Maslov grading preserving quasi-isomorphism $\psi_{\mathrm{flip}}: C\{i=0\}\to C\{j=0\}$. We will see that these algebraic conditions give us the necessary constraints on the flip-map.
\begin{proof}[Proof of Proposition \ref{flipmap}]
$\mathcal{\widetilde{B}}^R_g$ is defined as the dual knot of $-1$ surgery on $\mathcal{B}_g$. Hence, it suffices to illustrate that the condition on the flip-map is necessary for the dual knot of $+1$ surgery on $\mathcal{\widetilde{B}}^R_g$ to be the Borromean knot $\mathcal{B}_g$. From \cite{MattWatson}, a complex represents the Borromean knot iff there is a single generator in the top (respectively bottom) Alexander grading with no differentials mapping from (respectively to) it. Let us analyze the mapping cone for the dual knot of $+1$ surgery on $\mathcal{\widetilde{B}}^R_g$ in each of these extremal Alexander gradings, which in our case are $-g$ and $g$.
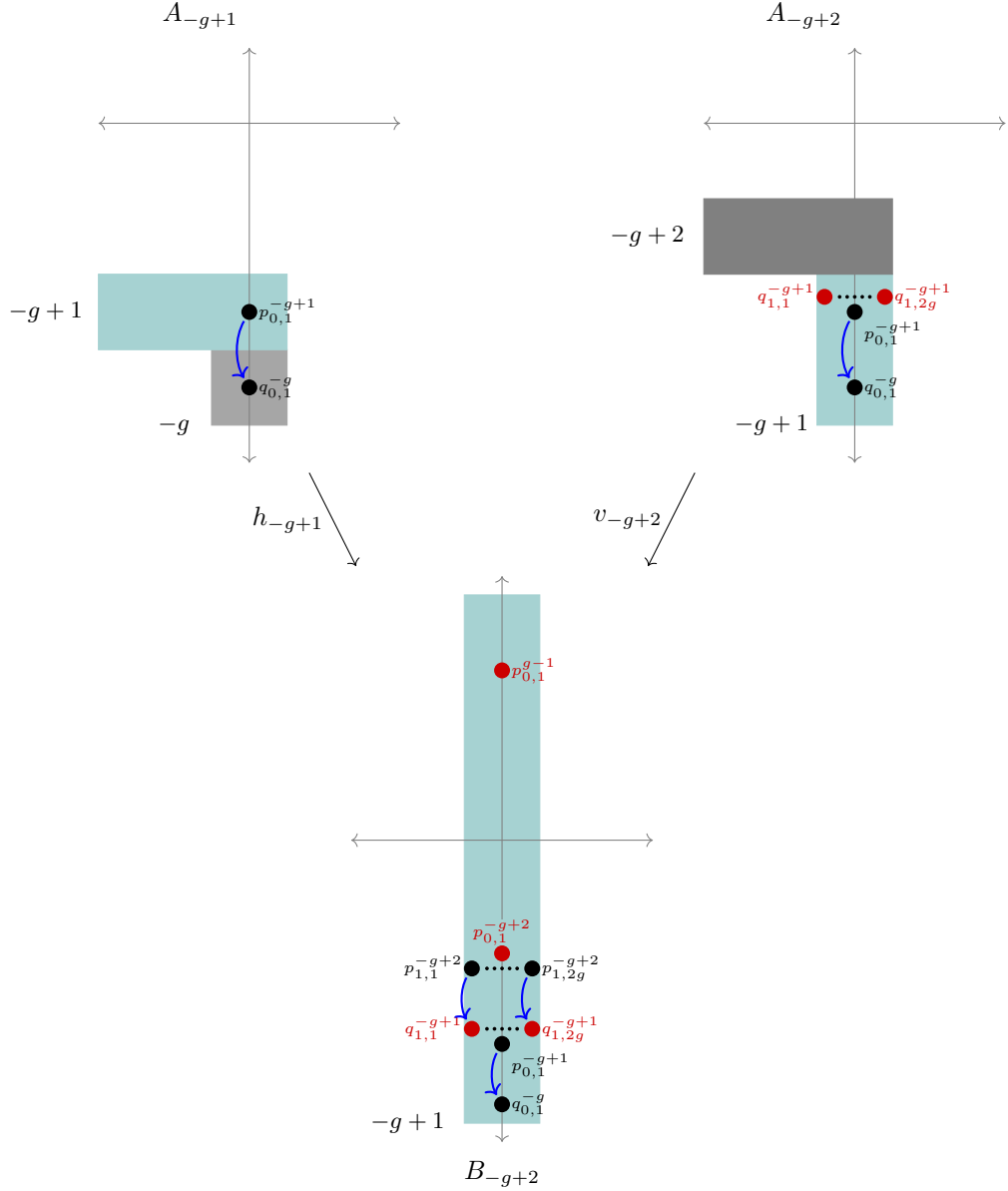
\begin{figure}[h!]
\centering
\begin{tikzpicture}

\node[label=south:$B_{-g+2}$] (B2) at (-1,1) {
\begin{tikzpicture}[scale=1]
    \filldraw[teal!35] (0, -3.25) rectangle (1, 3.75);

    \begin{scope}[thin, black!50!white]
        \draw[<->] (-1.5, 0.5) -- (2.5, 0.5);
        \draw[<->] (0.5, -3.5) -- (0.5, 4);
    \end{scope}

    \node at (-0.75,-3.25) {\small $-g+1$};
    
    \fill (0.5,-3) circle (3pt)node[ right, font=\tiny]{$q_{0,1}^{-g}$};
    
    \fill (0.5,-2.2) circle (3pt)node[below right, font=\tiny]{$p_{0,1}^{-g+1}$};
    \draw[->, blue,thick, bend right, shorten >=4pt, shorten <=4pt] (0.5, -2.2) to (0.5, -3);
    
    \fill[red!80!black] (0.1,-2) circle (3pt)node[ left, font=\tiny]{$q_{1,1}^{-g+1}$};
    \fill[red!80!black] (0.9,-2) circle (3pt)node[ right, font=\tiny]{$q_{1,2g}^{-g+1}$};
    \foreach \x in {0.3, 0.4, 0.5, 0.6, 0.7} {
    \fill[black] (\x, -2) circle (0.7pt);}

    \fill[red!80!black] (0.5,-1) circle (3pt) node[above, yshift=3pt, font=\tiny,fill=teal!35, inner sep=0]{$p_{0,1}^{-g+2}$};
    \fill (0.1,-1.2) circle (3pt)node[ left, font=\tiny]{$p_{1,1}^{-g+2}$};
    \fill (0.9,-1.2) circle (3pt)node[ right, font=\tiny]{$p_{1,2g}^{-g+2}$};
    \foreach \x in {0.3, 0.4, 0.5, 0.6, 0.7} {
    \fill[black] (\x, -1.2) circle (0.7pt);}
    \draw[->, blue,thick, bend right, shorten >=4pt, shorten <=4pt] (0.1, -1.2) to (0.1, -2);
    \draw[->, blue,thick, bend right, shorten >=4pt, shorten <=4pt] (0.9, -1.2) to (0.9, -2);
    
    \fill[red!80!black] (0.5,2.75) circle (3pt) node[right, font=\tiny]{$p_{0,1}^{g-1}$};
\end{tikzpicture}
};

\node[label=north:$A_{-g+1}$] (A1) at (-5,9) {
\begin{tikzpicture}[scale=1]
    \filldraw[black!35!white] (0, -2) rectangle (1, -1);
    \filldraw[teal!35] (-1.5, 0) rectangle (1, -1);

    \begin{scope}[thin, black!50!white]
        \draw[<->] (-1.5, 2) -- (2.5, 2);
        \draw[<->] (0.5, -2.5) -- (0.5, 3);
    \end{scope}

    \node at (-0.50,-2) {\small $-g$};
    \node at (-2.2,-0.5) {\small $-g+1$};
    
    \fill  (0.5,-1.5) circle (3pt) node [right, font=\tiny] {$q_{0,1}^{-g}$};
    
    \fill (0.5 , -0.5) circle (3pt) node [right, font=\tiny] {$p_{0,1}^{-g+1}$};
    
    \draw[->, blue,thick, bend right, shorten >=4pt, shorten <=4pt] (0.5, -0.5) to (0.5, -1.5);
    
\end{tikzpicture}
};

\node[label=north:$A_{-g+2}$] (A2) at (3,9) {
\begin{tikzpicture}[scale=1]
    \filldraw[teal!35] (0, -2) rectangle (1, 0);
    \filldraw[black!50!white] (-1.5, 1) rectangle (1, 0);

    \begin{scope}[thin, black!50!white]
        \draw[<->] (-1.5, 2) -- (2.5, 2);
        \draw[<->] (0.5, -2.5) -- (0.5, 3);
    \end{scope}
    
    \node at (-2.25,0.5) {\small $-g+2$};
    \node at (-0.60,-2) {\small $-g+1$};
    
    \fill[red!80!black] (0.1,-0.3) circle (3pt)node[ left, font=\tiny]{$q_{1,1}^{-g+1}$};
    \fill[red!80!black] (0.9,-0.3) circle (3pt)node[ right, font=\tiny]{$q_{1,2g}^{-g+1}$};
    \foreach \x in {0.3, 0.4, 0.5, 0.6, 0.7} {
    \fill[black] (\x, -0.3) circle (0.7pt);}

    \fill (0.5,-1.5) circle (3pt) node[right, font=\tiny] {$q_{0,1}^{-g}$};
    
    \fill (0.5 , -0.5) circle (3pt) node[below right, font=\tiny] {$p_{0,1}^{-g+1}$};
    
    \draw[->, blue,thick, bend right, shorten >=4pt, shorten <=4pt] (0.5, -0.5) to (0.5, -1.5);

\end{tikzpicture}
};

		\draw[->] (A1) to node[pos=0.5, left]{$h_{-g+1}$} (B2);
        \draw[->] (A2) to node[pos=0.5, left]{$v_{-g+2}$} (B2);
        
\end{tikzpicture}
\caption{Mapping cone for Alexander grading $-g+1$ highlighted in teal. The generators in red are the possible candidates for $h_{-g+1}(p_{0,1}^{-g+1})$. The blue arrows are internal differentials.}
\label{top}
\end{figure}

In the mapping cone for Alexander grading $-g$, there is a single generator $q_{0,1}^{-g}\in A_{-g+1}$. $q_{0,1}^{-g}$ is a boundary in $A_{-g+1}$ as $\partial(p_{0,1}^{-g+1})= q_{0,1}^{-g}$. See Figure \ref{top}. We need to ensure that $p_{0,1}^{-g+1}$ vanishes in the homology of the $-(g-1)$st mapping cone $\mathrm{MC}(-g+1)$ so that $q_{0,1}^{-g}$ is not a boundary in the mapping cone. The Maslov grading of $p_{0,1}^{-g+1}$ is $-g^2+1$. When we identify the horizontal strip $C\{j=-g+1\}$ with the vertical strip $C\{i=0\}$, there is an additional factor of $U^{-(g-1)}$. Since the flip-map is Maslov grading preserving, it must map $p_{0,1}^{-g+1}$ to a generator in the subcomplex $C\{i=0\}$ with Maslov grading $-g^2+2g-1$. Let us list the elements of $C\{i=0\}$ that have Maslov grading $-g^2+2g-1$ since these are precisely the possible images of $h_{-g+1}(p_{0,1}^{-g+1})$. There is one such element $p^{g-1}_{0,1}$ in Alexander grading $g-1$, there is another element $p^{-g+2}_{0,1}$ in Alexander grading $-g+2$ and there are $2g$ elements $q^{-g+1}_{1,1}, q^{-g+1}_{1,2},\ldots,q^{-g+1}_{1,2g}$ in Alexander grading $-g+1$. Now we show that regardless, $p_{0,1}^{-g+1}$ doesn't survive in the homology of the mapping cone $\mathrm{MC}(-g+1)$. Clearly, if the flip map sends $p_{0,1}^{-g+1}$ to an element with Alexander grading greater than $-g+1$, it cannot be in the image of $v_{-g+2}$ in the mapping cone $\mathrm{MC}(-g+1)$. This ensures that $p_{0,1}^{-g+1}$ vanishes in the  homology of $\mathrm{MC}(-g+1)$ since it cannot be a cycle. Also note that since these are not in the image of  $v_{-g+2}$, we do not get induced longer differentials after cancelling. For the remaining cases, observe that $p_{0,1}^{-g+1}$ survives when you take the homology in the horizontal strip, while $q^{-g+1}_{1,i}$'s have non-trivial differential on them in the vertical strip. Since the flip map is a quasi-isomorphism, $\psi_{\mathrm{flip}}$ cannot map $p_{0,1}^{-g+1}$ to these elements. Hence, we see that in all cases, $p_{0,1}^{-g+1}$ won't survive in the homology of the mapping cone for Alexander grading $-g+1$. Hence in the lowest Alexander grading, we have a single generator that is not in the image of the differential.

\begin{figure}[h!]
\centering
\begin{tikzpicture}

\node[label=south:$B_{g-2}$] (B1) at (-5,1) {
\begin{tikzpicture}[scale=1]
    \filldraw[blue!35] (0, -2) rectangle (1, 3);

    \begin{scope}[thin, black!50!white]
        \draw[<->] (-1.5, 0) -- (2.5, 0);
        \draw[<->] (0.5, -2.25) -- (0.5, 3.25);
    \end{scope}
    
    \node at (-0.5,-2) {\small $g-3$};
    
    \fill  (0.5,0.5) circle (3pt) node [right, font=\small] {$x_{g-2}'$};

\end{tikzpicture}
};

\node[label=south:$B_{g-2}$] (B2) at (0,1) {
\begin{tikzpicture}[scale=1]
    \filldraw[green!35] (0, -2) rectangle (1, 3);

    \begin{scope}[thin, black!50!white]
        \draw[<->] (-1.5, 0) -- (2.5, 0);
        \draw[<->] (0.5, -2.25) -- (0.5, 3.25);
    \end{scope}
    
    \node at (-0.5,-2) {\small $g-2$};
       
    \fill  (0.5,1.5) circle (3pt) node [right, font=\small] {$x_{g-1}'$};

\end{tikzpicture}
};

\node[label=south:$B_{g}$] (B3) at (5,1) {
\begin{tikzpicture}[scale=1]
    \filldraw[teal!35] (0, -2) rectangle (1, 3);

    \begin{scope}[thin, black!50!white]
        \draw[<->] (-1.5, 0) -- (2.5, 0);
        \draw[<->] (0.5, -2.25) -- (0.5, 3.25);
    \end{scope}
    
    \node at (-0.5,-2) {\small $g-1$};
    
    \fill  (0.5,2.5) circle (3pt) node [right, font=\small] {$x_g'$};
    
\end{tikzpicture}
};

\node[label=north:$A_{g-2}$] (A1) at (-5,8) {
\begin{tikzpicture}[scale=1]
    \filldraw[blue!35] (0, -1.5) rectangle (1, 0);
    \filldraw[green!35] (-1.5, 0) rectangle (1, 1);

    \begin{scope}[thin, black!50!white]
        \draw[<->] (-1.5, 0) -- (2.5, 0);
        \draw[<->] (0.5, -1.75) -- (0.5, 3.25);
    \end{scope}
    
    \fill  (0.5,0.5) circle (3pt) node [right, font=\small] {$x_{g-2}:=q^{g-2}_{0,1}$};

\end{tikzpicture}
};

\node[label=north:$A_{g-1}$] (A2) at (0,8) {
\begin{tikzpicture}[scale=1]
    \filldraw[green!35](0, -1.5) rectangle (1, 1);
    \filldraw[teal!35] (-1.5, 1) rectangle (1, 2);

    \begin{scope}[thin, black!50!white]
        \draw[<->] (-1.5, 0) -- (2.5, 0);
        \draw[<->] (0.5, -1.75) -- (0.5, 3.25);
    \end{scope}
    
    \fill  (0.5,1.5) circle (3pt) node [right, font=\small] {$x_{g-1}:=q^{g-1}_{0,1}$};

\end{tikzpicture}
};

\node[label=north:$A_{g}$] (A3) at (5,8) {
\begin{tikzpicture}[scale=1]
   \filldraw[teal!35] (0, -1.5) rectangle (1, 2);
   \filldraw[yellow!35]  (-1.5, 2) rectangle (1, 3);

    \begin{scope}[thin, black!50!white]
        \draw[<->] (-1.5, 0) -- (2.5, 0);
        \draw[<->] (0.5, -1.75) -- (0.5, 3.25);
    \end{scope}
    
    \node at (-1.75,2.5) {\small $g$};
    
    \fill  (0.5,2.5) circle (3pt) node [right, font=\small] {$x_g:=q^g_{0,1}$};

\end{tikzpicture}
};

		\draw[->] (A2) to node[pos=0.5, left]{$h_{g-1}$} (B3);
        \draw[->] (A1) to node[pos=0.5, left]{$h_{g-2}$} (B2);
        \draw[->] (A1) to node[pos=0.5, left]{$v_{g-2}$} (B1);
		\draw[->] (A2) to node[pos=0.5, left]{$v_{g-1}$} (B2);
		\draw[->] (A3) to node[pos=0.5, left]{$v_{g}$} (B3);
\end{tikzpicture}
\caption{Mapping cone for dual knot of +1 surgery on $\mathcal{\widetilde{B}}^R_g$.}
\label{bot}
\end{figure}
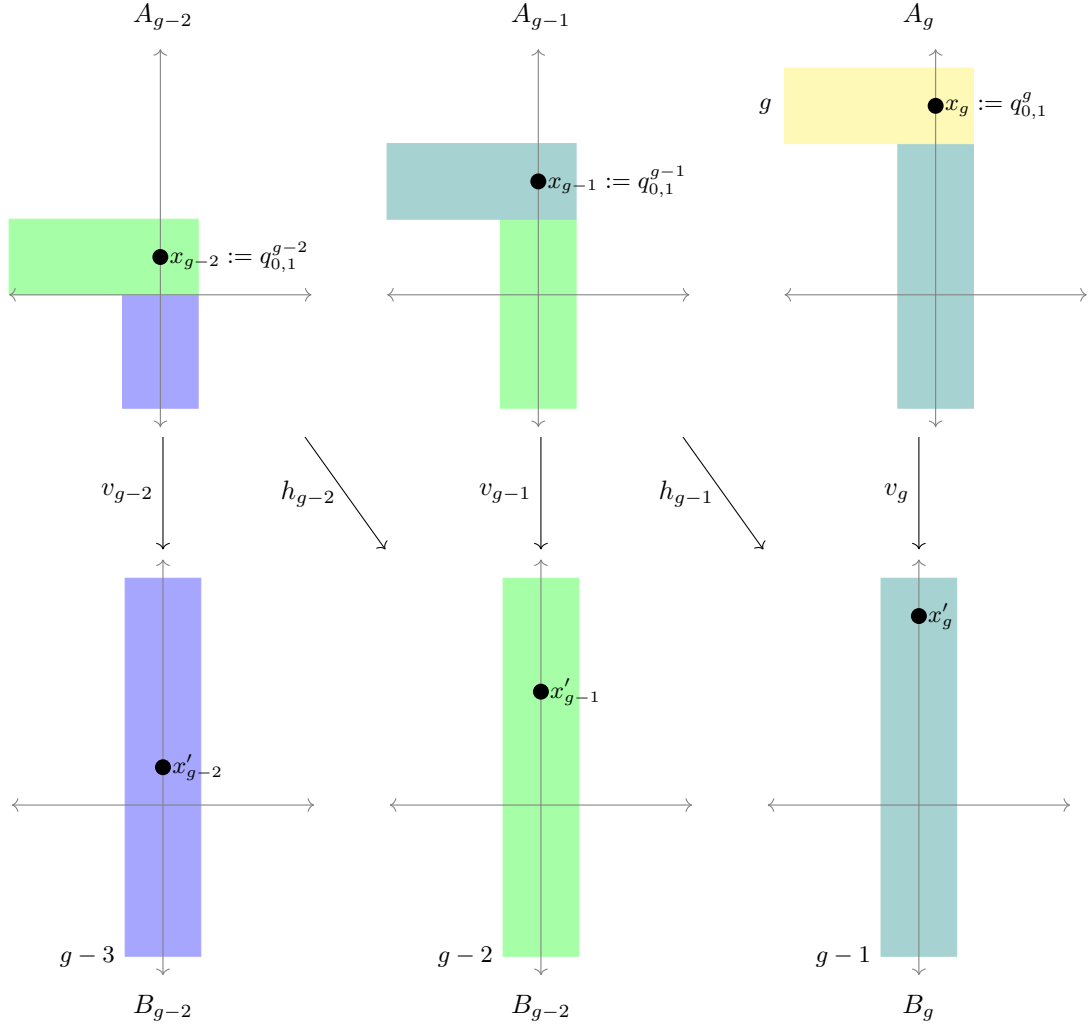

Let us now analyze the case of Alexander grading $g$. It is clear that in the mapping cone for  Alexander grading $g$ that there is a single generator $q_{g}^{0,1}\in A_g$. See Figure \ref{bot}. Let us denote this by $x_g$. The differential vanishes on $x_g$ in the mapping cone if its image under $v_g$ is a boundary, i.e. $\partial ^\infty(x_g)=v_g(x_g)=x_g'\in B_g$ is a boundary.  The Maslov grading of $x_g'$ is $-g^2+2g$. Observing the elements of the horizontal strip that survive in the homology of the mapping cone for Alexander grading $g-1$ and also map to $x_g'$ under $h_{g-1}$, there is only a single element $q_{0,1}^{g-1}$ of degree $-g^2+4g-2$, and indeed this is a cycle in the mapping cone. Denote this by $x_{g-1}$. We are forced to have $h_{g-1}(x_{g-1})=x_g'$. We are not quite done, since $\partial^\infty(x_{g-1})=h_{g-1}(x_{g-1})+v_{g-1}(x_{g-1})=x_g'+v_{g-1}(x_{g-1})$. So we need to show that $v_{g-1}(x_{g-1})=x_{g-1}'$ is a boundary. This boils down to finding an element $x_{g-2}\in A_{g-2}$ such that $h_{g-2}(x_{g-2})=v_{g-1}(x_{g-1})=x_{g-1}'$. Again, looking at Maslov degrees and elements in $A_{g-2}$ that survive in the mapping cone for Alexander grading $g-2$, there is a unique element $q_{0,1}^{g-2}$. This forces the flip map on $h_{g-2}(x_{g-2})=x_{g-1}'$. We keep repeating this process, observing that each step we only have one possible choice, which forces the $h$ map (which is basically the flip-map), and we end up with the following subcomplex. 
\[
\begin{tikzcd}[column sep=15pt, row sep=10pt]
	{\bullet x_{-(g-1)}} && {\bullet x_{-(g-2)}} && {\bullet x_{-(g-3)}} && {\bullet x_{g-2}} && {\bullet x_{g-1}} && {\bullet x_g} \\
	\\
	&&&&& \ldots \\
	&& {\bullet x'_{-(g-2)}} && {\bullet x'_{-(g-3)}} &&&& {\bullet x'_{g-1}} && {\bullet x_g'}
	\arrow["{{h_{-(g-1)}}}"{pos=0.4}, from=1-1, to=4-3]
	\arrow["{{v_{-(g-2)}}}"{pos=0.6}, from=1-3, to=4-3]
	\arrow["{{h_{-(g-2)}}}"{pos=0.4}, from=1-3, to=4-5]
	\arrow["{{v_{-(g-3)}}}"{pos=0.6}, from=1-5, to=4-5]
	\arrow["{{h_{g-2}}}"{pos=0.4}, from=1-7, to=4-9]
	\arrow["{{v_{g-1}}}"{pos=0.6}, from=1-9, to=4-9]
	\arrow["{{h_{g-1}}}"{pos=0.4}, color=blue, from=1-9, to=4-11]
	\arrow["{{v_g}}"{pos=0.6}, from=1-11, to=4-11]
\end{tikzcd}
\]
We can cancel the blue arrow to get a long red differential as follows.
\[
\begin{tikzcd}[column sep=15pt, row sep=10pt]
	{\bullet x_{-(g-1)}} && {\bullet x_{-(g-2)}} && {\bullet x_{-(g-3)}} && {\bullet x_{g-2}} &&  && {\bullet x_g} \\
	\\
	&&&&& \ldots \\
	&& {\bullet x'_{-(g-2)}} && {\bullet x'_{-(g-3)}} &&&& {\bullet x'_{g-1}}
	\arrow["{{h_{-(g-1)}}}"{pos=0.4}, from=1-1, to=4-3]
	\arrow["{{v_{-(g-2)}}}"{pos=0.6}, from=1-3, to=4-3]
	\arrow["{{h_{-(g-2)}}}"{pos=0.4}, from=1-3, to=4-5]
	\arrow["{{v_{-(g-3)}}}"{pos=0.6}, from=1-5, to=4-5]
	\arrow["{{h_{g-2}}}"{pos=0.4}, from=1-7, to=4-9]
	\arrow["{{\widetilde{v_g}}}"{pos=0.6},color=red, from=1-11, to=4-9]
\end{tikzcd}
\]

Iteratively repeating this process by canceling $h_{g-i}$, we end up with only $x_g$. Hence, in the end, when we compute the complex of the dual knot, the differential of $x_g$ is zero. This completes the proof. 
\end{proof}

Note that we have proved something stronger. Following the proof, we can also conclude that the conditions above on the flip map are not only necessary but also sufficient for the dual knot of $+1$ surgery on $\mathcal{\widetilde{B}}^R_g$ to be $\mathcal{B}_g$.

Proceeding as above, we can prove an analogous statement for the flip-map of $\mathcal{\widetilde{B}}^L_g$.

\begin{prop}
\label{flipmap2}
Let $\Theta_{k} \in C\{0,k\} \subseteq \mathrm{CFK}^{\infty}(\mathcal{\widetilde{B}}^L_g)$ denote the unique generator of $P(g,k,0)$ for Alexander grading $-g+1\leq k \leq g$. Then $\psi_{\mathrm{flip}}: C\{j=0\}\to C\{i=0\}$ is given by $\psi_{\mathrm{flip}}(U^{k}\Theta_k)=\Theta_{k-1}$.
\end{prop}

\section{Monodromy Detection}\label{section-proof}

In this section, we prove that the complex $\mathrm{CFK}^{\infty}(\mathcal{\widetilde{B}}^R_g)$ detects $\mathcal{\widetilde{B}}^R_g$.

\begin{proof}[Proof of Theorem \ref{detect}]
The strategy is as follows: Let $K\subseteq Y$ be any knot whose complex is given as in Theorem \ref{MT} and satisfies the flip-map constraint as in Proposition \ref{flipmap}. Consider the dual knot $K^\nu$ of $+1$ surgery on $K$. We saw in the previous section that with the constraints on the flip-map, $K^\nu$ is isotopic to $\mathcal{B}_g$. Now observe that the dual knot of $-1$ surgery on $K^\nu$ is isotopic to $K$. Hence, we conclude that $K$ is isotopic to $\mathcal{\widetilde{B}}^R_g$.
\end{proof}

We also have the following theorem for $\mathcal{\widetilde{B}}^L_g$.
\begin{theorem}
\label{detect2}
Let $K\subset Y$ be an arbitrary knot in some three-manifold $Y$. Suppose that $\mathrm{CFK}^{\infty}(K,Y)$ is filtered chain homotopic to $\mathrm{CFK}^{\infty}(\mathcal{\widetilde{B}}^L_g)$. Furthermore, assume that the flip map on  $\mathrm{CFK}^{\infty}(K,Y)$ satisfies Proposition \ref{flipmap2}. Then $K$ is isotopic to $\mathcal{\widetilde{B}}^L_g$.
\end{theorem}

\bibliographystyle{alpha}
\bibliography{references}

\end{document}